%% file: manuscript.tex
\documentclass[a4paper,oneside]{article}
\usepackage{fullpage}

\input{preamble.tex}

\newcommand{\TheAuthorADM}{Alberto De~Marchi}
\newcommand{\TheEmailADM}{alberto.demarchi@unibw.de}
\newcommand{\TheOrcidADM}{0000-0002-3545-6898}
\newcommand{\TheAffiliationADM}{%
	University of the Bundeswehr Munich,
	Department of Aerospace Engineering,
	Institute of Applied Mathematics and Scientific Computing,
	85577 Neubiberg, Germany%
}
\newcommand{\TheAcknowledgementsADM}{%
	The author thanks the anonymous reviewers for their stimulating comments, which led to intriguing observations and eventually stronger arguments.
}
\newcommand{\TheShortTitle}{Mixed-integer linearity in nonlinear optimization}
\newcommand{\TheTitle}{\TheShortTitle:\\a trust region approach}
\newcommand{\TheKeywords}{%
	Mixed-integer programming\KeywordsAnd
	nonlinear programming\KeywordsAnd
	successive linearization schemes\KeywordsAnd
	trust region methods\KeywordsEnd
}
\newcommand{\TheAMSsubj}{%
	\amsmscLink{65K05}\AMSand
	\amsmscLink{90C06}\AMSand
	\amsmscLink{90C11}\AMSand
	\amsmscLink{90C30}
}

\newcommand{\TheAbstract}{%
	Bringing together nonlinear optimization with polyhedral and integrality constraints
	enables versatile modeling, but poses significant computational challenges.
	We investigate a method to address these problems based on sequential mixed-integer linearization with trust region safeguard,
	computing feasible iterates via calls to a generic mixed-integer linear solver.
	Convergence to critical, possibly suboptimal, feasible points is established for arbitrary starting points.
	Finally, we present numerical applications in nonsmooth optimal control and optimal network design and operation.%
}

\begin{document}
	
\title{\bfseries \TheTitle}
\author{\TheAuthorADM\thanks{\TheAffiliationADM.
		\emailLink{\TheEmailADM},
		\orcidLink{\TheOrcidADM}.%
	}}

\maketitle
\begin{abstract}
	\TheAbstract%
	\subclass{\TheAMSsubj}%
\end{abstract}

\tableofcontents

\section{Introduction}

\fancyinitial{T}{he} operation of complex, integrated process systems demands efficient use of resources
whilst imposing tight safety constraints.
Mixed-integer optimization provides a powerful and flexible mathematical template for modeling many tasks that involve discrete and continuous variables \cite{belotti2013mixed,grossmann2013systematic}.
This work lies at the intersection of nonlinear and integer programming,
leaning toward a continuous optimization perspective,
and explores how nonlinear programming (NLP) techniques can help as heuristics.
We focus here on the minimization of a smooth function subject to polyhedral constraints, an important problem class in NLP, complemented by integrality specifications.
The denomination ``mixed-integer linearity'' in the title refers to the mixed nature of real- and integer-valued decision variables and to the linear relationships restricting their admissible values.
In particular,
we consider optimization problems with a nonlinear objective and mixed-integer linear constraints in the form
\begin{align}\tag{P}\label{eq:P}
	\minimize\quad{}& f(x) \\
	\wrt\quad{}& x\in\XX
	\coloneqq
	\bar{\XX}
		\cap
	\left\{
	x \in \spaceX
	\,\middle|\,
	x_i \in \Z \quad \forall i\in\II
	\right\} \nonumber
\end{align}
with objective function $\func{f}{\XX}{\R}$ and nonempty feasible set $\XX \subseteq \spaceX$.
The latter is described by a closed convex polyhedral set $\bar{\XX}\subseteq\spaceX$ and integrality constraints defined by some index set $\II \subseteq \{1,\ldots,n\}$.
We intend to exploit the tractable, although difficult, structure of $\XX$ in our numerical method,
in particular avoiding quadratic terms that may arise from local models or regularizations
and instead relying on mixed-integer linear programs (MILP) as subproblems.
To do so, we follow \cite[Section 3]{dambrosio2012storm} and make use of polyhedral norms, e.g., $\ell_p$ with $p\in\{1,\infty\}$, instead of the classical Euclidean $\ell_2$.
Then, a practical assumption is that MILPs involving $\XX$ can be efficiently solved,
leveraging mature LP and MILP technology,
despite their NP-completeness.

We investigate \eqref{eq:P} with the implicit assumption that there are both real- and integer-valued variables as well as some couplings thereof in $\XX$ and nonlinearities in $f$.
Should this not be the case, off-the-shelf methods for MILP and NLP could be readily adopted for tackling \eqref{eq:P}.
Then,
partitioning the decision variables $x$ into real-valued $u$ and integer-valued $z$,
we consider hereafter some blanket assumptions on \eqref{eq:P}.
First of all, we suppose that
\begin{enumerate}[label=({A\arabic{*}})]
	\item\label{ass:dfLipschitz}%
	function $f$ is continuously differentiable
	and $\nabla f$ is locally Lipschitz continuous;
	\item\label{ass:flinearz}%
	the objective function $f$
	is of the form $(u,z) \mapsto f_1(u) + \innprod{f_2}{z}$,
	for some smooth function $f_1$ and vector $f_2$ of suitable size;
	\item\label{ass:zbounded}%
	all feasible values for the integer-valued variables $z$ lie in a bounded set.
\end{enumerate}
Assumption~\ref{ass:dfLipschitz} concerns the objective function $f$ and its smoothness, and requires only local Lipschitz gradient continuity \cite{demarchi2023monotony}.
Assumption~\ref{ass:flinearz} of integer-linearity for $f$ 
is introduced, in line with \cite{quirynen2021sequential}, to rule out
the need for derivative approximations or exotic notions; e.g., the difference formula in \cite{exler2007trust}.
Notice, however, that Assumption~\ref{ass:flinearz} is considered for clarity of presentation and it is not restrictive.
A generic nonlinear, albeit smooth, dependence of $f$ on the integer-valued variables can always be incorporated
with a suitable reformulation that satisfies Assumption~\ref{ass:flinearz},
possibly adding auxiliary real-valued variables and complementing $\XX$ with linear equality constraints.
Specification \ref{ass:zbounded}
is not only reasonable and often valid in practice \cite{belotti2013mixed}, the prominent modeling tool being binary variables,
but it also greatly simplifies the presentation, with minor algorithmic consequences.

Focusing on \eqref{eq:P} without any convexity assumptions,
we will not demand global optimality
but seek instead an affordable local ``solution''.
The underlying optimality notion is defined
by localization in a trust region.
We analyze the similarities of projected gradient-type methods with trust region linearization schemes,
highlighting how polyhedral norms may lead to undesirable properties for the former,
making the latter an interesting alternative.
After detailing an iterative numerical scheme,
we investigate its global convergence properties, in the sense of arbitrary initial points.
The algorithm generates feasible iterates with (nonmonotonically) decreasing objective,
and relies on averaged merit values to monitor and enforce convergence.

\paragraph*{Related Work}

The ideas behind sequential linearization and trust region methods are not new \cite{yuan1985conditions,byrd2005convergence}, but there has been little progress toward relaxing some fundamental assumptions.
Sequential linearization methods for problems with polyhedral constraints heavily exploit convexity of the feasible set \cite{byrd2005convergence},
notwithstanding that the approach has been applied to particular nonconvex problems as well,
prominently by \cite{kirches2022sequential} for complementarity constraints.
In contrast, proximal algorithms have been developed for the fully nonconvex structured setting,
with globalization based on line search \cite{themelis2018forward}
and trust region \cite{aravkin2022proximal}.
However, since they rely on local quadratic models to achieve sufficient decrease,
their direct application to \eqref{eq:P} is hindered in practice.

\paragraph*{Contributions and Outline}
We trace back the sequential linearization approach
and overcome the issues arising from the feasible set being nonconvex.
This is made possible by treating the integralities through a black-box,
both theoretically and numerically.
\Cref{sec:optimality}
is dedicated to optimality concepts for \eqref{eq:P} and focuses on
a notion of criticality that
requires access only to a linear minimization oracle.
We propose in \cref{sec:smil} a numerical method based on sequential mixed-integer linearization,
with a globalization mechanism that brings together trust region and line search.
The algorithm is characterized and global convergence results are established under mild assumptions.
Finally, in \cref{sec:numresults} we report numerical results on two nontrivial example problems, illustrating the algorithmic behavior and performance.

\section{Approach and Optimality Concepts}\label{sec:optimality}

\fancyinitial{T}{his} section is dedicated to juxtaposing projection and trust region approaches to tackle \eqref{eq:P}.
Optimality notions and characterizations are then discussed in view of the numerical algorithm developed later on.

\subsection{From projections to trust regions}

Attracted by the success of proximal-gradient methods to numerically solve nonconvex nonsmooth problems,
one may be tempted to address \eqref{eq:P} with such methods recast with polyhedral norms.
The classical \emph{round} setting involves an $\ell_2$-norm minimization to compute a projection onto the feasible set $\XX$,
whereas polyhedral projections based on $\ell_1$- or $\ell_\infty$-norm can be reformulated in terms of MILP.
A round projected gradient update at $\bar{x}\in\XX$ is given by
\begin{equation}
	\label{eq:projgrad0}
	x^+
	\in
	\proj_\XX( \bar{x} - \gamma \nabla f(\bar{x}) )
	=
	\argmin_{x \in \XX} \frac{1}{2 \gamma}\| \bar{x} - \gamma \nabla f(\bar{x}) - x \|_2^2
\end{equation}
for some stepsize $\gamma > 0$.
Owing to the structure of $\XX$,
the round projection requires solving a mixed-integer \emph{quadratic} program (MIQP), because of the quadratic term.
Seeking a MILP subproblem,
one could replace the Euclidean norm with a polyhedral one as distance metric
and update according to
\begin{equation}\label{eq:lpProjGrad}
	x^+ \in \argmin_{x \in \XX} \norm{ \bar{x} - \gamma \nabla f(\bar{x}) - x }_p .
\end{equation}
Unfortunately,
taking such projections with $p\in\{1,\infty\}$ can lead to odd situations.
In fact,
not only polyhedral projections can be set-valued even for convex sets,
but a (badly selected) projected gradient may yield an increase of the objective, regardless of the stepsize.

Moving away from the projection point-of-view,
we can consider a linear model of $f$ around $\bar{x}$ and,
to avoid artificial unboundedness,
replace the quadratic regularization term
inherent in \eqref{eq:projgrad0}
with a trust region constraint.
Such trust region subproblem at $\bar{x}\in\XX$ with radius $\Delta > 0$ reads
\begin{equation}\label{eq:criticality_milp}
	\minimize_{x\in\XX} \quad f(\bar{x}) + \innprod{\nabla f(\bar{x})}{x - \bar{x}}
	\qquad
	\stt\quad \normlp{x - \bar{x}} \leq \Delta ,
\end{equation}
where $\func{\normlp{\cdot}}{\spaceX}{\R_+}$ denotes the application of a norm on the real-valued entries only, according to the index set $\II$.%
\footnote{%
	If the index set $\II$ is nonempty, $\normlp{\cdot}$ lacks positive definiteness, since $\normlp{x}=0 \notimplies x=0$, and therefore it is not a norm. However, $\normlp{\cdot}$ is subadditive and absolutely homogeneous, namely it satisfies
	\[\normlp{x+y} \leq \normlp{x}+\normlp{y}
	\qquad\text{and}\qquad
	\normlp{\alpha x} = |\alpha| \normlp{x}
	\]
	for all $x,y\in\spaceX$ and $\alpha\in\R$.
}
Despite the trust region stipulation being relative only to the real-valued variables, hence the notation ``PL'' for ``partial localization'',
\eqref{eq:criticality_milp} is well posed owing to Assumption~\ref{ass:zbounded},
which makes the intersection of $\XX$ with (closed) \emph{PL-balls}
\[
\lpball(x,\Delta) \coloneqq \left\{ w\in\spaceX \,|\, \normlp{w-x}\leq\Delta \right\}
\]
always a compact set.
Instead, the partial localization with $\normlp{\cdot}$ gives rise to a more inclusive definition of neighborhoods, and hence a comparatively strong concept of ``local optimality'', see \cref{sec:optimalityConcepts} below.

Since the construction of \eqref{eq:criticality_milp} is similar in spirit to \eqref{eq:projgrad0},
one can expect a connection between the two approaches.
For an $\ell_p$-projected gradient scheme, $1 \leq p \leq \infty$,
the next iterate from a feasible point $\bar{x}\in\XX$ is obtained from
\eqref{eq:lpProjGrad},
whose solutions
$x^+ \in \XX$ must satisfy
$
	\norm{ x^+ - \bar{x} + \gamma \nabla f(\bar{x}) }_p
	\leq
	\gamma \norm{ \nabla f(\bar{x}) }_p
$
and, using the triangle inequality,
\begin{equation*}
	\norm{x^+ - \bar{x}}_p
	\leq{}
	\norm{x^+ - \bar{x} + \gamma \nabla f(\bar{x})}_p + \norm{- \gamma \nabla f(\bar{x})}_p 
	\leq{}
	2 \gamma \norm{\nabla f(\bar{x})}_p .
\end{equation*}
This means that
a projected gradient subproblem imposes a trust region-like condition,
whereas a trust region subproblem does not seek a projection but instead to minimize (a local model of) the objective function.
Thus, as on a quest to avoid round terms, we explore the more natural trust region perspective with polyhedral norms.
More precisely, aiming at MILP subproblems, we adopt a polyhedral norm for constructing $\normlp{\cdot}$, deliberately excluding round norms such as the classical $\ell_2$-norm and the weighted norm $\|\cdot\|_M$ for some $M\succ 0$.
Therefore, possible definitions for $\normlp{\cdot}$ are $\normlp{x} \coloneqq \sum_{i\notin\II} |x_i|$ and $\normlp{x} \coloneqq \max_{i\notin\II} |x_i|$, respectively, for the $\ell_1$- and $\ell_\infty$-norm.

\subsection{Optimality concepts}\label{sec:optimalityConcepts}

Considering the problem of minimizing a continuous real function $f$ over a closed set
$\XX$, it is clear what a global solution is:
we call $\bar{x} \in \XX$ a \emph{global minimizer} for \eqref{eq:P} if $f(\bar{x}) \leq f(x)$ for all $x \in \XX$.
In this context, however, the notion of \emph{local} minimizer remains elusive, because of the delicate role played by neighborhoods, which can easily make \emph{all} isolated points (of $\XX$) local minimizers.
For constructing a useful localized optimality concept,
we consider the notion of neighborhoods with partial localization induced by
the $\normlp{\cdot}$, as adopted in \eqref{eq:criticality_milp}.

\begin{definition}\label{def:isolated}
	A \emph{PL-neighborhood} of a point $\bar{x}\in\XX$ is a subset of $\spaceX$ that contains an open PL-ball $\left\{ w\in\spaceX \,|\, \normlp{w-\bar{x}} < \Delta \right\}$ for some $\Delta>0$.
	A point $\bar{x}\in\XX$ is called a \emph{PL-isolated} point of $\XX$ if there exists
	a PL-neighborhood of $\bar{x}$ that does not contain any other point of $\XX$.
\end{definition}
Notice that for PL-isolated points of $\spaceX$, \cref{def:isolated} reduces to the classical notion of isolatedness.

This gives rise to the following definition,
consistent with the global counterpart.

\begin{definition}[minimizers]\label{def:minimizer}
	A point $\bar{x}\in\XX$ is called a \emph{local minimizer} for \eqref{eq:P} if
	there exists $\Delta>0$ such that
	$f(\bar{x}) \leq f(x)$
	for all $x\in\XX\cap\lpball(\bar{x},\Delta)$.
	If this property holds for all $\Delta>0$,
	then $\bar{x}$ is called a \emph{global minimizer} for \eqref{eq:P}.
\end{definition}

Seeking affordable algorithms to find candidate solutions for \eqref{eq:P},
we focus on strong (and yet practical) necessary conditions for such (local) optimality notion
\cite[Chapter 3]{birgin2014practical}.
According to \cite[Section 3]{themelis2018forward}, $\bar{x}$ is \emph{critical} if it is $\gamma$-critical for some $\gamma > 0$, i.e., $\bar{x} \in \proj_\XX( \bar{x} - \gamma \nabla f(\bar{x}) )$,
and such notion is stronger than stationarity for nonconvex problems.
Due to the difficulties discussed above in relation to round projections, we shall consider a notion based on \eqref{eq:criticality_milp} instead.
Thus,
we define the criticality measure $\psimeas$ associated to \eqref{eq:P} as follows:
\begin{equation}
	\label{eq:psimeas}
	\forall x\in\XX ,\,
	\forall \Delta\geq 0 \colon
	\quad
	\psimeas(x; \Delta)
	\coloneqq
	\max\left\{
	\innprod{\nabla f(x)}{x-w}
	\,|\,
	w\in\XX\cap\lpball(x,\Delta)
	\right\}
	.
\end{equation}
Since $x\in\XX$, the lower bound $\psimeas(\cdot;\Delta) \geq 0$ holds for all $\Delta \geq 0$.
Although analogous to the criticality measure of \cite{yuan1985conditions,byrd2005convergence}, which focused on \eqref{eq:P} without integrality constraints, in \eqref{eq:psimeas} $\psimeas$ depends additionally on the trust region radius $\Delta$, which provides a localization---dispensable in the convex setting.
Considering (global) minimizers of the program obtained by linearization of $f$ about $\bar{x}$ and localization with a trust region, \eqref{eq:psimeas} leads to the following definition.

\begin{definition}[criticality]\label{def:criticality}
	Given some $\Delta > 0$,
	a point $\bar{x} \in \XX$ is called \emph{$\Delta$-critical} for \eqref{eq:P} if $\psimeas(\bar{x}; \Delta) = 0$;
	it is called \emph{critical} for \eqref{eq:P} if it is $\Delta$-critical for some $\Delta>0$.
\end{definition}

Remark that \cref{def:criticality} is well-posed: if there exists some $\bar{\Delta}>0$ such that $\psimeas(\bar{x};\bar{\Delta}) = 0$, namely $x = \bar{x}$ is optimal with $\bar{\Delta}$,
then $\bar{x}$ is necessarily optimal for all $\Delta \in [0, \bar{\Delta}]$,
for otherwise it could not be optimal for the relaxed constraint with $\bar{\Delta}$.

\cref{def:criticality} is closely related to stationarity in nonlinear optimization,
interpreted as the absence of feasible first-order descent directions,
but localized as criticality in (nonconvex) structured optimization \cite{themelis2018forward}.
The following foundational yet simple result states that criticality constitutes indeed a necessary first-order optimality condition for \eqref{eq:P},
in the sense of \cref{def:minimizer,def:criticality}, respectively.
\begin{proposition}\label{prop:necessary}
	Let $\bar{x}\in\spaceX$ be a local minimizer for \eqref{eq:P}.
	Then, $\bar{x}\in\XX$ is critical.
\end{proposition}
\begin{proof}
	Feasibility of $\bar{x}$ is readily inherited.
	Owing to Assumptions~\ref{ass:flinearz}--\ref{ass:zbounded},
	optimality and criticality coincide with regard to the integer-valued variables. 
	Then, the assertion is elementary if $\nabla f(\bar{x}) = 0$ or $\bar{x}\in\XX$ is PL-isolated.
	Otherwise,
	continuous differentiability of $\func{f}{\XX}{\R}$ guarantees that no local improvements are possible (see, e.g., \cite[Lemma 3.2]{birgin2014practical}),
	concluding the proof.
\end{proof}

\cref{prop:necessary} portrays criticality as a necessary condition for (local and global) optimality,
regardless of constraint qualifications.
An analogous concept is that of B-stationarity:
a point $\bar{x}\in\XX$ is called \emph{B-stationary} for \eqref{eq:P} if
$- \nabla f(\bar{x}) \in \normalconefrechet_\XX(\bar{x})$,
where $\normalconefrechet_\XX(\bar{x})$ denotes the regular (or Fr\'echet) normal cone of $\XX$ at $\bar{x}$ \cite[Def. 6.38]{bauschke2017convex}.
A relevant stationarity notion in nonlinear optimization, B-stationarity can be interpreted as the absence of feasible first-order descent directions,
see \cite{kirches2022sequential}.
Although B-stationarity and criticality coincide for \eqref{eq:P} in the special case $\II=\emptyset$,
the following illustrative example with complementarity constraints shows that, in general, B-stationarity $\centernot{\implies}$ criticality.
A more comprehensive characterization of criticality is left for future elaborations, possibly tailored to prominent problem classes, such as NLP with combinatorial structures.

\begin{example}
	Consider \eqref{eq:P} with $\XX\subset \R^2\times\Z$ described by
	\begin{align*}
	0 \leq u_1 \leq U z,\quad
	0 \leq u_2 \leq U (1-z),\quad
	z\in\{0,1\}
	\end{align*}
	for some fixed $U > 0$.
	Notice that the feasible set $\XX$ provides a lifted mixed-integer linear description of the nonconvex set $\mathcal{C}\subset\R^2$ obtained by intersecting the complementarity constraint $0 \leq u_1 \perp u_2 \geq 0$ with the hyperbox $[0,U]^2$.
	The notation $u_1 \perp u_2$ here stands for the requirement $u_1 u_2 = 0$.
	The problematic point $\bar{u} = (0,0)$,
	in its ambient space $\R^2$,
	is B-stationary if and only if $-\nabla f_1(\bar{u}) \in \normalconefrechet_{\mathcal{C}}(\bar{u}) = (-\infty,0]^2$.	

	Now we focus on $\bar{x}_0 = (\bar{u},0)$ and $\bar{x}_1 = (\bar{u},1)$, the only two points of $\XX$ recovering $\bar{u}$ when projected onto the plane $(u_1,u_2)$.
	Elaborating on \cref{def:criticality} we find that
	the problematic points $\bar{x}_0$ and $\bar{x}_1$ are critical if and only if $-\nabla f(\bar{x}_0) \in (-\infty,0]^3$ and $-\nabla f(\bar{x}_1) \in (-\infty,0]^2\times[0,\infty)$, respectively.
	If the objective $f$ does not depend on the binary variable $z$ (that is, if $f_2=0$), then the conditions for B-stationarity of $\bar{u}$ in $\mathcal{C}$ and those for criticality of $\bar{x}_0$ and $\bar{x}_1$ in $\XX$ coincide.
	In contrast, if $f_2\ne 0$, so that $f$ explicitly depends on $z$, then criticality is strictly more restrictive than B-stationarity, which corresponds to
	$-\nabla f(\bar{x}_0) \in \R\times(-\infty,0]\times\R$ and $-\nabla f(\bar{x}_1) \in (-\infty,0]\times\R^2$.
\end{example}

Before developing and analyzing a numerical scheme for \eqref{eq:P},
we shall introduce an approximate counterpart of \cref{def:criticality}.
The following definition requires that a point $\bar{x}\in\XX$ is not too far from being a minimizer for \eqref{eq:criticality_milp},
as monitored using $\psimeas$ from \eqref{eq:psimeas}.

\begin{definition}[$\varepsilon$-criticality]\label{def:approx_criticality}
	Given some $\varepsilon>0$ and $\Delta > 0$,
	a point $\bar{x} \in \XX$ is called \emph{$\varepsilon$-$\Delta$-critical} for \eqref{eq:P} if $\psimeas(\bar{x}; \Delta) \leq \varepsilon$;
	it is called \emph{$\varepsilon$-critical} for \eqref{eq:P} if it is $\varepsilon$-$\Delta$-critical for some $\Delta>0$.
\end{definition}

For later reference we give the following characterization of noncritical points.
\begin{lemma}\label{lem:noncritical}
	Let $\bar{x} \in \XX$ be not critical for \eqref{eq:P}.
	Then, $\nabla f(\bar{x}) \neq 0$,
	$\bar{x}$ is not a PL-isolated point of $\XX$ and,
	for all $\Delta>0$, there exist $x\in\XX\cap\lpball(\bar{x},\Delta)$ and $\varepsilon>0$ such that $\innprod{\nabla f(\bar{x})}{x-\bar{x}} \leq - \varepsilon$.
	Moreover, for any given $\delta>0$ the inequalities
	\begin{equation}
		\label{eq:sandwich}
		\psimeas(\bar{x};\delta)
		\geq
		\psimeas(\bar{x};\Delta)
		\geq
		\frac{\Delta}{\delta} \psimeas(\bar{x};\delta)
	\end{equation}
	hold for all $\Delta\in[0,\delta]$.
\end{lemma}
\begin{proof}
	If $\nabla f(\bar{x}) = 0$ or $\bar{x}$ is PL-isolated, then it follows from \eqref{eq:psimeas} that $\bar{x}\in\XX$ must be critical according to \cref{def:criticality}.
	This proves the first two assertions.
	The third is obtained similarly, by negating the criticality property for $\bar{x}$.
	
	Let us consider the sandwich inequalities in \eqref{eq:sandwich}.
	Pick an arbitrary but fixed $\delta>0$ and let $\Delta\in(0,\delta]$.
	The remaining case $\Delta=0$ is immediate since $\psimeas(\bar{x};0)=0$ and $\psimeas(\cdot;\cdot) \geq 0$.
	Since $\Delta \leq \delta$, the left inequality directly follows from the definition in \eqref{eq:psimeas}.
	Before addressing the inequality on the right, we rewrite \eqref{eq:psimeas} in the convenient form
	\begin{equation*}
		\forall x\in\XX ,\,
		\forall \Delta\geq 0 \colon
		\quad
		\psimeas(x;\Delta)
		\coloneqq
		\max\left\{
		-\innprod{\nabla f(x)}{d}
		\,|\,
		x+d\in\XX ,\, \normlp{d}\leq\Delta
		\right\}
	\end{equation*}
	and we denote by $d_\Delta$ a (global) maximizer in this definition of $\psimeas(x;\Delta)$, for any given $\Delta\geq0$.
	Then, by construction of $d_\delta$, it must be $\normlp{d_{\delta}} \leq \delta$ and, owing to homogeneity of $\normlp{\cdot}$, we obtain the bound $\normlp{d_\delta \Delta/\delta} \leq \Delta$.
	Finally, together with \eqref{eq:psimeas}, this bound implies the following inequality
	\begin{equation*}
		\psimeas(\bar{x};\Delta)
		=
		- \innprod{\nabla f(\bar{x})}{d_\Delta}
		\geq
		- \innprod{\nabla f(\bar{x})}{d_\delta \Delta/\delta}
		=
		- \frac{\Delta}{\delta} \innprod{\nabla f(\bar{x})}{d_\delta}
		=
		\frac{\Delta}{\delta} \psimeas(\bar{x};\delta),
	\end{equation*}
	concluding the proof.
\end{proof}

In contrast to \cite[Lemma 3.1]{byrd2005convergence},
which is valid for problems with convex constraints,
the noncriticality assumption in \cref{lem:noncritical} is essential for the sandwich inequalities \eqref{eq:sandwich} in our context.
Counterexamples are readily obtained by taking a point $\bar{x}$ that is $\Delta$-critical but not $\delta$-critical for some $\delta > \Delta > 0$.

\section{Sequential mixed-integer linearization method}\label{sec:smil}

\fancyinitial{W}{e} now introduce our iterative numerical scheme based on local mixed-integer linear approximations.
Then, we investigate its behavior and show the convergence to critical points under mild assumptions.

\subsection{Algorithm}\label{sec:algorithm}

The proposed method, detailed in \cref{alg:SMIL}, builds upon a sequence of MILP subproblems of the form \eqref{eq:criticality_milp}.
After solving the trust region linearized subproblem at \cref{step:SMIL:subproblem},
which guarantees that $x^{k+1} \in \XX \cap \lpball(x^k,\Delta_k)$ and
\begin{equation}
	\label{eq:psimeasAlgorithm}
	\psimeas(x^k;\Delta_k)
	=
	\psimeas_k
	\coloneqq
	\innprod{\nabla f(x^k)}{x^k - x^{k+1}}
	\geq
	0,
\end{equation}
the quality of the tentative update $x^{k+1}$ from $x^k$ is tested. 
In particular, we monitor criticality and terminate at \cref{step:SMIL:termination} if (approximately) satisfied. 
Then, a sufficient decrease condition requires at \cref{step:SMIL:suffreduction} that
$
	f(x^{k+1})
	\leq
	\meritf_k - \varrho \psimeas(x^k;\Delta_k)
$
holds for the update $x^{k+1}$ to be accepted;
otherwise, similarly to a backtracking line search procedure, the trust region radius $\Delta_k$ is reduced and the attempt discarded.
With the lens of trust region methods,
\cref{step:SMIL:suffreduction} compares the predicted linearized improvement against the actual change,
with the ratio parameter $\varrho>0$ acting as threshold for acceptance.
In fact, the actual change $a_k$ at \cref{step:SMIL:actualChange} refers to a merit value $\meritf_k$ which coincides with the cost only if $\meritpfac=1$, owing to \cref{step:SMIL:meritUpdate}.
For $\meritpfac\in(0,1)$ the acceptance criterion is less conservative and potentially allows larger radii and faster convergence in practice,
with the objective $f(x^k)$ decreasing nonmonotonically;
see \cite{demarchi2023monotony}.
Finally, the radius is updated at \cref{step:SMIL:radiusUpdate},
where one can apply classical trust region rules or a line search-like reset,
based on $\rho_k \coloneqq a_k / \psimeas_k$,
such as
\begin{equation}
	\label{eq:TRupdate}
	\Delta_{k+1}
	\coloneqq
	\begin{cases}
		\kappa \Delta_k &\text{if}~\rho_k < \varrho_1 ,\\
		\Delta_k &\text{if}~\varrho_1 \leq \rho_k < \varrho_2 ,\\
		\Delta_k/\kappa &\text{if}~\varrho_2 \leq \rho_k \\
	\end{cases}
	\qquad\text{or}\qquad
	\Delta_{k+1}
	\in [\Delta_{\min}, \Delta_{\max}]
	,
\end{equation}
respectively,
given parameters such that $\varrho \leq \varrho_1 < \varrho_2 < 1$ and $0 < \Delta_{\min} \leq \Delta_{\max}$.

\begin{algorithm2e}[tb]
	\DontPrintSemicolon
	\KwIn{$x^0 \in \XX$, $\varepsilon > 0$}
	\KwData{$\Delta_0>0$,
		$\varrho, \kappa \in (0,1)$,
		$\meritpfac\in(0,1]$}
	set $\meritf_0 \gets f(x^0)$\;
	\For{$k = 0,1,2\ldots$}{
		compute $x^{k+1} \in \argmin \left\{ \innprod{\nabla f(x^k)}{x} \,|\, x \in \XX \cap \lpball(x^k,\Delta_k) \right\}$\label{step:SMIL:subproblem}\;
		set $a_k \gets \meritf_k - f(x^{k+1})$ and
		$\psimeas_k \gets \innprod{\nabla f(x^k)}{x^k - x^{k+1}}$\label{step:SMIL:actualChange}\label{step:SMIL:predictChange}\;
		\lIf{$\psimeas_k \leq \varepsilon$\label{step:SMIL:termination}}{%
			\KwReturn $x^k$\algcomment{$\varepsilon$-$\Delta_k$-critical}%
		}
		\If{$a_k < \varrho \psimeas_k$\label{step:SMIL:suffreduction}}{%
			set $\Delta_k \gets \kappa \Delta_k$ and \KwGoTo \cref{step:SMIL:subproblem}\label{step:SMIL:backtrack}
		}
		set $\meritf_{k+1} \gets (1-\meritpfac) \meritf_k + \meritpfac f(x^{k+1})$\label{step:SMIL:meritUpdate}\;
		select $\Delta_{k+1}$ by trust region update \eqref{eq:TRupdate}\label{step:SMIL:radiusUpdate}\;
	}
	\caption{Successive mixed-integer linearization algorithm for \eqref{eq:P}}
	\label{alg:SMIL}
\end{algorithm2e}

We note that the input requirement in \cref{alg:SMIL} of a feasible initial point is not a restriction:
given some $x^0 \notin \XX$, it is enough to project it onto the feasible set by solving a MILP, for instance
$\min_{x\in\XX} \| x - x^0 \|_1$, and run the algorithm from a solution thereof.

\begin{remark}\label{rem:nlp_refinement}
An optional refinement step can also be integrated in \cref{alg:SMIL},
similarly to magic steps \cite[Section 8.2]{birgin2014practical},
equality-constrained quadratic programming (EQP) steps of \cite{byrd2005convergence},
and bound-constrained quadratic programming (BQP) steps of \cite{kirches2022sequential}.
One can also consider NLP steps, by fixing the integer-valued variables and solving an NLP with polyhedral constraints \cite{dambrosio2012storm,belotti2013mixed}.
Without impairing the convergence properties,
this optional step can improve the objective value while maintaining feasibility,
possibly exploiting additional smoothness of $f$,
as discussed in \cref{sec:numresults} below.
\end{remark}

Finally, the termination criterion at \cref{step:SMIL:termination} is based on the approximate criticality notion of \cref{def:approx_criticality},
and supported by the following convergence analysis, which guarantees finite termination for any $\varepsilon > 0$ when $f$ is bounded from below over $\XX$.

\subsection{Convergence analysis}

Assumption~\ref{ass:dfLipschitz} on the smoothness of $f$ is sufficient to prove that the backtracking due to \cref{step:SMIL:backtrack} eventually stops.
In particular, the condition at \cref{step:SMIL:suffreduction} is violated for a sufficiently small $\Delta_k$, after finitely many attempts.
Consequently, \cref{alg:SMIL} is well-defined and produces a sequence of feasible iterates $\{x^k\}$ along with decreasing merit values $\{\meritf_k\}$;
cf. \cref{lem:descentBehavior}.
Then, it remains to prove that the accumulation points are also
critical.
After the asymptotic characterization in \cref{lem:MILasymptotics},
we show that,
in a neighborhood of a feasible but noncritical point,
the subproblem at \cref{step:SMIL:subproblem} eventually generates an acceptable step
while the trust region radius remains bounded away from zero.
Finally, \cref{thm:convergence2critical} synthesizes the main convergence results.

We start with a preliminary result about the reduction implied by \cref{step:SMIL:subproblem},
which follows from Taylor's theorem.
Notice how Assumption~\ref{ass:flinearz} allows to drop the integer-valued variables from the quadratic term of the expansion.

\begin{lemma}\label{lem:partialTaylor}
	Consider
	a function $\func{f}{\XX}{\R}$, a nonempty compact set $\Omega \subseteq \dom f$,
	and suppose Assumptions~\ref{ass:dfLipschitz}--\ref{ass:zbounded} hold. 
	Partitioning $x=(u,z)$ with $u$ and $z$ the real- and integer-valued variables, respectively,
	define
	\begin{equation*}
		L \coloneqq
		\sup
		\left\{
		\frac{\|\nabla f(x_1) - \nabla f(x_2)\|_1}{\|u_1 - u_2\|_\infty}
		\,\middle|\,
		\begin{aligned}  
			&x_1, x_2 \in \Omega , \\
			&u_1 \neq u_2
		\end{aligned}
		\right\} .
	\end{equation*}
	Then $L < \infty$, and for all $x_1, x_2 \in \Omega$ the following upper bound is valid:
	\begin{equation*}
		f(x_2)
		\leq
		f(x_1) + \innprod{\nabla f(x_1)}{x_2-x_1} + \frac{L}{2} \|u_2-u_1\|_{\infty}^2
		.
	\end{equation*}
\end{lemma}

\begin{proof}
	By Assumption~\ref{ass:flinearz} it is $\nabla f(x) = ( \nabla f_1(u), f_2 )$,
	hence $L<\infty$ coincides with the Lipschitz constant of $\nabla f$ over $\Omega$, finite by Assumption~\ref{ass:dfLipschitz}.
	Then, the asserted inequality follows from Taylor-type arguments,
	such as in the descent lemma \cite[Lemma 2.64]{bauschke2017convex},
	by considering the variation from $x_1$ to $x_2$ along real- and integer-valued components separately.
	We sketch the second part of the proof as follows.
	Define $\widehat{x} \coloneqq (u_2, z_1)$ and, for simplicity, assume $\widehat{x}\in\Omega$.
	Then, Assumption~\ref{ass:flinearz} and Taylor's theorem, e.g. \cite[Lemma 4.3]{kirches2022sequential}, yield the inequality
	\begin{equation*}
		f(x_2)
		=
		f(\widehat{x}) + \innprod{\nabla f(\widehat{x})}{x_2-\widehat{x}}
		\\
		\leq
		f(x_1) + \innprod{\nabla f(x_1)}{\widehat{x}-x_1} + \frac{L}{2} \|\widehat{x}-x_1\|_\infty^2 + \innprod{\nabla f(\widehat{x})}{x_2-\widehat{x}} .
	\end{equation*}
	Finally, exploiting the structure of $\nabla f$ and the definition of $\widehat{x}$, one recognizes the asserted upper bound.
\end{proof}

Next, we employ \cref{lem:partialTaylor} to prove that the iterations of \cref{alg:SMIL} are well-defined.

\begin{lemma}[well-definedness]\label{lem:welldefinedness}
	Consider \eqref{eq:P} under Assumptions~\ref{ass:dfLipschitz}--\ref{ass:zbounded}.
	Then the condition at \cref{step:SMIL:suffreduction} of \cref{alg:SMIL} is violated after a finite number of steps.
\end{lemma}

\begin{proof}
	If $x^k$ is critical, then it is $\bar{\Delta}$-critical for some $\bar{\Delta}>0$,
	and $x = x^k$ solves the subproblem at \cref{step:SMIL:subproblem} as soon as $\Delta_k \leq \bar{\Delta}$.
	Regardless of the choice $x^{k+1} \in\XX$, it must yield $\psimeas_k = 0$ and immediate termination of the algorithm, for any $\varepsilon>0$.
	
	If $x^k$ is not critical, then we rely on the sandwich inequalities \eqref{eq:sandwich} given by \cref{lem:noncritical}.
	The claim will follow from the fact that $\meritf_0 = f(x^0)$ and, for each $k\in\N$,
	the acceptance according to \cref{step:SMIL:suffreduction} and the update rule at \cref{step:SMIL:meritUpdate} imply that
	\begin{equation}
		\label{eq:lowerBoundMerit}
		\meritf_{k+1}
		\coloneqq{}
		(1 - \meritpfac) \meritf_k + \meritpfac f(x^{k+1})
		\\
		\geq{}
		(1 - \meritpfac) \left[ f(x^{k+1}) + \varrho \psimeas_k \right] + \meritpfac f(x^{k+1})
		\geq{}
		f(x^{k+1})
		,
	\end{equation}
	proving that the nonmonotone acceptance criterion is less conservative than the monotone one
	(which has $\meritpfac=1$ and $\meritf_k = f(x^k)$ for all $k\in\N$).
	Moreover,
	all attempts satisfy $x^{k+1} \in \XX$ and $\normlp{x^{k+1} - x^k} \leq \Delta_k$.
	In particular, all attempts $x^{k+1}$ remain in a bounded set,
	owing to Assumption~\ref{ass:zbounded} and to the trust region stipulation,
	so there exists a constant $L_k > 0$ related to $\nabla f$ over that set, in the sense of \cref{lem:partialTaylor}.
	Then, all attempts $x^{k+1}$ satisfy the upper bound
	\begin{equation}\label{eq:lips_bound}
		f(x^{k+1}) 
		\leq
		f(x^k) + \innprod{\nabla f(x^k)}{x^{k+1} - x^k} + \frac{L_k}{2} \|u^{k+1} - u^k\|_{\infty}^2
		.
	\end{equation}
	Combined with \eqref{eq:lowerBoundMerit} and then \eqref{eq:psimeasAlgorithm}, \eqref{eq:lips_bound} yields
	\begin{multline}
		\label{eq:rho_bound}
		\rho_k
		\coloneqq
		\frac{a_k}{\psimeas_k}
		=
		\frac{m_k - f(x^{k+1})}{\psimeas_k}{}
		\geq{}
		\frac{f(x^k) - f(x^{k+1})}{\psimeas_k}
		\geq{}
		\frac{\psimeas_k - \frac{L_k}{2} \|u^{k+1} - u^k\|_{\infty}^2}{\psimeas_k} \\
		=
		1 - \frac{L_k}{2} \frac{\|u^{k+1} - u^k\|_{\infty}^2}{\psimeas_k}
		\geq{}
		1 - \frac{L_k \delta_k}{2} \frac{\|u^{k+1} - u^k\|_{\infty}^2}{\Delta_k \psimeas(x^k;\delta_k)}
		,
	\end{multline}
	where the last inequality follows from
	the lower bound in \eqref{eq:sandwich}, for any fixed, sufficiently large $\delta_k>0$ and the associated $\psimeas(x^k;\delta_k)>0$.
	Due to the equivalence of norms in $\spaceX$, there exists a constant
	$c_{\textup{PL}} > 0$ relating $\normlp{x}$ and $\|u\|_\infty$, providing a bound in the form $\|u^{k+1} - u^k\|_\infty \leq c_{\textup{PL}} \normlp{x^{k+1}-x^k}$.
	Therefore, from \eqref{eq:rho_bound} we have that
	\begin{align*}
		2 (1 - \rho_k) \frac{\psimeas(x^k;\delta_k)}{L_k \delta_k}
		\leq
		\frac{\|u^{k+1} - u^k\|_{\infty}^2}{\Delta_k}
		\leq
		\frac{c_{\textup{PL}}^2 \normlp{x^{k+1} - x^k}^2}{\Delta_k}
		\leq
		c_{\textup{PL}}^2 \Delta_k
		,
	\end{align*}
	where the last step uses the trust region stipulation with radius $\Delta_k$ due to \cref{step:SMIL:subproblem} of \cref{alg:SMIL}.
	Since $x^k$, $L_k$ and $\delta_k$ are fixed
	while $\Delta_k$ is decreased,
	the condition $\rho_k \geq \varrho$ becomes satisfied as soon as
	\begin{equation}
		\label{eq:welldefinedness:DeltaThreshold}
		\Delta_k
		\leq
		2\frac{1-\varrho}{c_{\textup{PL}}^2}\frac{\psimeas(x^k;\delta_k)}{L_k \delta_k}
		,
	\end{equation}
	violating the condition at \cref{step:SMIL:suffreduction},
	after finitely many backtracks since $\kappa\in(0,1)$.
\end{proof}

\begin{lemma}[descent behavior]\label{lem:descentBehavior}
	Consider \eqref{eq:P} under Assumptions~\ref{ass:dfLipschitz}--\ref{ass:zbounded}
	and the iterates generated by \cref{alg:SMIL}.
	Then,
	\begin{enumerate}[label=(\roman{*})]
		\item\label{lem:descentBehavior:1}%
		the sequence $\{\meritf_k\}$ is monotonically decreasing and, for every $k\in\N$, one has
		$\psimeas_k \geq 0$ and
		\begin{equation}\label{eq:descentBehavior}
			f(x^{k+1}) + (1-\meritpfac) \varrho \psimeas_k
			\leq
			\meritf_{k+1}
			\leq
			\meritf_k - \meritpfac \varrho \psimeas_k
			.
		\end{equation}
		\item\label{lem:descentBehavior:2}%
		Every iterate $x^k$ is feasible, namely $x^k\in\XX$, and satisfies $f(x^k) \leq f(x^0)$.
	\end{enumerate}
\end{lemma}

\begin{proof}
	Based on \cref{lem:welldefinedness}, the iterates of \cref{alg:SMIL} are well-defined.
	For \ref{lem:descentBehavior:1}, by the update rule at \cref{step:SMIL:meritUpdate} and the condition at \cref{step:SMIL:suffreduction}, we have that
	\begin{align*}
		\meritf_{k+1}
		\coloneqq{}
		(1-\meritpfac) \meritf_k + \meritpfac f(x^{k+1})
		\leq{}
		(1-\meritpfac) \meritf_k + \meritpfac \left[ \meritf_k - \varrho \psimeas_k \right]
		={}
		\meritf_k - \meritpfac \varrho \psimeas_k
	\end{align*}
	holds upon acceptance.
	The lower bound on $\meritf_{k+1}$ is readily obtained from \eqref{eq:lowerBoundMerit}.
	Then, it follows from \eqref{eq:psimeasAlgorithm} that $\psimeas_k\geq 0$ for each $k\in\N$.
	
	Let us consider \ref{lem:descentBehavior:2} now.
	Owing to $x^0 \in \XX$ and \cref{step:SMIL:subproblem}, every iterate $x^k$ is feasible by construction.
	Then, since $\meritpfac \in(0,1]$ and $\varrho \psimeas_k\geq 0$, it follows from $\meritf_0 \coloneqq f(x^0)$ and \ref{lem:descentBehavior:1} that
	$f(x^k) \leq \meritf_k \leq \ldots \leq \meritf_0 \coloneqq f(x^0) < \infty$.
\end{proof}

\begin{lemma}\label{lem:MILasymptotics}
	Consider \eqref{eq:P} under Assumptions~\ref{ass:dfLipschitz}--\ref{ass:zbounded}
	and the iterates generated by \cref{alg:SMIL}.
	Suppose $\inf_{\XX} f > -\infty$.
	Then,
	\begin{enumerate}[label=(\roman{*})]
		\item\label{lem:MILasymptotics:1}%
		$\{f(x^k)\}$ and $\{\meritf_k\}$ converge, the latter from above, to a finite value $f_\star \geq \inf_{\XX} f$;
		\item\label{lem:MILasymptotics:2}%
		$\sum_{k\in\N} \psimeas_k < \infty$;
		\item\label{lem:MILasymptotics:limPsimeas} $\lim_{k\to\infty} \psimeas_k = 0$.
	\end{enumerate}
\end{lemma}

\begin{proof}
	Based on \cref{lem:welldefinedness}, the iterates of \cref{alg:SMIL} are well-defined
	and we may assume the sequence $\{x^k\}$ is infinite (it suffices to skip \cref{step:SMIL:termination}).
	The assertion in \ref{lem:MILasymptotics:1} regarding $\{\meritf_k\}$ follows from \eqref{eq:descentBehavior} with the premises.
	Then, considering the update rule at \cref{step:SMIL:meritUpdate},
	\begin{equation*}
		\meritpfac f(x^{k+1})
		=
		\meritf_{k+1} - (1-\meritpfac) \meritf_k
		=
		\meritpfac \meritf_k + (\meritf_{k+1} - \meritf_k)
		,
	\end{equation*}
	the convergence of $\{\meritf_k\}$ implies that of $\{f(x^k)\}$ to the same value $f_\star$, since $\meritpfac\in(0,1]$.
	
	Regarding \ref{lem:MILasymptotics:2}, a telescoping argument on \eqref{eq:descentBehavior}, together with \ref{lem:MILasymptotics:1}, yields
	\begin{equation*}
		\meritf_0 - f_\star
		\geq
		\sum_{k=0}^j \meritf_k - \meritf_{k+1}
		\geq
		\meritpfac \varrho \sum_{k=0}^j \psimeas_k
		\geq
		0
		.
	\end{equation*}
	The finite sum follows from the independence of the (finite) upper bound $\meritf_0-f_\star$ on $j\in\N$.
	Finally, assertion \ref{lem:MILasymptotics:limPsimeas} follows from \ref{lem:MILasymptotics:2}, since $\psimeas_k \geq 0$ by \cref{lem:descentBehavior}.
\end{proof}

Owing to \eqref{eq:psimeasAlgorithm},
\cref{lem:MILasymptotics}\ref{lem:MILasymptotics:limPsimeas} gives that $\{ \psimeas(x^k;\Delta_k) \} \to 0$ if $\inf_{\XX} f>-\infty$.
Therefore, if $\{\Delta_k\}$ remains bounded away from zero, we expect accumulation points of $\{x^k\}$ to be critical.
Indeed,
for any noncritical point $x^\infty\in\XX$ and sequence $\{x^k\}$ such that $x^k\to x^\infty$, the trust region radius $\{\Delta_k\}$ can be bounded from below by some positive constant.
In the following result, inspired by \cite[Lemma 4.5]{kirches2022sequential}, by \emph{relative PL-neighborhood} we mean the intersection of a PL-neighborhood of $x^\infty$ in $\spaceX$ with the feasible set $\XX$.

\begin{lemma}\label{lem:noncriticalBoundedRadius}
	Consider \eqref{eq:P} under Assumptions~\ref{ass:dfLipschitz}--\ref{ass:zbounded}.
	Let a point $x^\infty \in \XX$ be not critical for \eqref{eq:P}.
	Then, there exists
	a nonsingleton relative PL-neighborhood $N^\infty$ of $x^\infty$
	such that for any sequence $\{x^k\}\subset N^\infty$ with $x^k \to x^\infty$
	the steps of \cref{alg:SMIL} generate a sequence $\{\Delta_k\}$ bounded away from zero.
\end{lemma}
\begin{proof}
	Owing to noncriticality and \cref{lem:noncritical}, $x^\infty$ is not a PL-isolated point of $\XX$, hence there exists
	a nonempty nonsingleton relative PL-neighborhood $N^\infty$ of $x^\infty$.
	Moreover, the sandwich inequalities \eqref{eq:sandwich} hold for any fixed $\delta_\infty>0$ and the associated $\psimeas(x^\infty;\delta_\infty)>0$.
	Denote $L_\infty$ the Lipschitz constant of $\nabla f$ over $N^\infty$,
	finite by Assumption~\ref{ass:dfLipschitz}.
	Since the trust region radius is reduced when needed for sufficient decrease, and not otherwise,
	$\Delta_k$
	is not reduced below the uniform bound given by $\kappa \Delta_\infty$,
	where
	\begin{equation*}
	\Delta_\infty
	\coloneqq
	2\frac{1-\varrho}{c_{\textup{PL}}^2}\frac{\psimeas(x^\infty;\delta_\infty)}{L_\infty \delta_\infty}
	>
	0
	\end{equation*}
	is the worst-case threshold for acceptance at \cref{step:SMIL:suffreduction};
	see \eqref{eq:welldefinedness:DeltaThreshold}.
	Analogous arguments apply with the reset rule for $\Delta_{k+1}$ in \eqref{eq:TRupdate}, concluding the proof.
\end{proof}

By virtue of \cref{lem:noncriticalBoundedRadius},
the trust region radius does not vanish in the vicinity of noncritical points,
hence sufficiently small steps can always be accepted there,
which deliver sufficient reduction and improvement.
We are now in a position to prove our main convergence result,
connecting accumulation points and criticality.
The following \cref{thm:convergence2critical} is akin to \cite[Thm 3.8]{byrd2005convergence} and \cite[Thm 4.2]{kirches2022sequential}.

\begin{theorem}\label{thm:convergence2critical}
	Consider \eqref{eq:P} under Assumptions~\ref{ass:dfLipschitz}--\ref{ass:zbounded}
	and let $x^0 \in \XX$ be an arbitrary but fixed point.
	Then, relative to the iterates generated by \cref{alg:SMIL} with $\varepsilon=0$,
	one of the following mutually exclusive outcomes must occur:
	\begin{enumerate}[label=(\roman{*})]
		\item\label{thm:convergence2critical:1}%
		$\lim_{k\to\infty} f(x^k) = -\infty$;
		\item\label{thm:convergence2critical:2}%
		\cref{alg:SMIL} terminates at a critical point, that is, $x^k$ solves the subproblem at \cref{step:SMIL:subproblem} and $\psimeas_k = 0$ for some $k\in\N$;
		\item\label{thm:convergence2critical:3}%
		\cref{alg:SMIL} generates an infinite sequence of iterates $\{ x^k \}$  with decreasing merit values $\{\meritf_k\}$.
		If $\{x^k\}$ has an accumulation point $x^\infty$, then $x^\infty$ is feasible and critical for \eqref{eq:P}.
	\end{enumerate}
\end{theorem}

\begin{proof}
	We need to consider only outcome \ref{thm:convergence2critical:3} because in the other cases we either observe \ref{thm:convergence2critical:1} lower unboundedness or \ref{thm:convergence2critical:2} obtain a critical point by virtue of the termination condition.

	Since the backtracking terminates finitely by \cref{lem:welldefinedness},
	and each attempt
	is feasible for the subproblem at \cref{step:SMIL:subproblem},
	we obtain by induction over the iterations that $x^k\in\XX$ for all $k\in\N$.
	Furthermore, by \cref{lem:descentBehavior}
	all iterations yield an improvement in terms of merit
	---and the objective value eventually decreases as well, owing to \cref{lem:MILasymptotics}.
	It remains to show that every accumulation point of $\{x^k\}$ is critical.

	Seeking a contradiction, we assume that an accumulation point $x^\infty$ of the sequence $\{x^k\}$ is not critical.
	Let us denote $\{x^k\}_{k\in K}$ a subsequence such that
	$x^k \to_K x^\infty$
	and, possibly relabeling, restrict $\{x^k\}_{k\in K}$ to the PL-neighborhood $N^\infty$ defined in \cref{lem:noncriticalBoundedRadius},
	so that $x^k \in N^\infty$ for all $k\in K$.

	Now we show that $\{\psimeas_k\}_{k\in K}$ remains bounded away from zero.
	First of all, by noncriticality of $x^\infty$,
	\cref{lem:noncritical} gives that $\innprod{\nabla f(x^\infty)}{x^\infty - x} \geq \varepsilon_0$ for some $x \in \XX$ and $\varepsilon_0 > 0$.
	Then, recall the continuity of $\nabla f$, the minimality property of $x^{k+1}$ for $\innprod{\nabla f(x^k)}{\cdot - x^k}$ within a radius $\Delta_k$ (bounded away from zero by \cref{lem:noncriticalBoundedRadius}), and the convergence $x^k \to x^\infty$.
	These facts imply the existence of some $\varepsilon_{\infty} > 0$ such that $\psimeas_k \geq \varepsilon_{\infty}$ holds for all $k \in K$.	
	Combined with \eqref{eq:descentBehavior},
	and for any $k_0\in K$, the boundedness of $\{\psimeas_k\}_{k\in K}$ away from zero results in
	\begin{multline*}
		f(x^k)
		\leq{}
		\meritf_k
		={}
		\meritf_{k_0} + \sum_{j=k_0}^{k-1} (\meritf_{j+1} - \meritf_j)
		\leq{}
		\meritf_{k_0} - \sum_{j=k_0}^{k-1} \meritpfac \varrho \psimeas_j
		\\
		\leq{}
		\meritf_{k_0} - \sum_{j=k_0}^{k-1} \meritpfac \varrho \varepsilon_{\infty}
		={}
		\meritf_{k_0} - (k-k_0) \meritpfac \varrho \varepsilon_{\infty}
		\to
		- \infty
	\end{multline*}
	as $k \to_K \infty$.
	This fact contradicts the (subsequential) convergence $f(x^k) \to_K f(x^\infty)$ by
	$x^k \to_K x^\infty$,
	thus showing that every accumulation point $x^\infty$ of $\{x^k\}$ is critical.
\end{proof}

Finally, as a direct corollary to \cref{thm:convergence2critical},
finite termination with an $\varepsilon$-critical point can be established for any $\varepsilon > 0$, if $\inf_{\XX} f > - \infty$, via \cref{lem:MILasymptotics}.

\section{Numerical Examples}\label{sec:numresults}

\fancyinitial{W}{e} report on numerical results for two example problems:
the discrete-time optimal control of a hybrid system with hysteresis
and the concurrent design and operation planning of a processing network.

To exploit the problem structure and the available technology, we consider also a variant of \cref{alg:SMIL} with the NLP refinement step suggested in \cref{rem:nlp_refinement}.
In practice,
the real-valued variables are perfected by solving the NLP obtained from \eqref{eq:P} by fixing the integer-valued variables at their current value.
To avoid solving many NLPs, this procedure is invoked only when the integer-valued variables remain untouched by a successful update.
More precisely, the NLP refinement may take place before executing \cref{step:SMIL:meritUpdate}, only if $z^{k+1} = z^k$.
An improved $u^{k+1}$ is obtained by solving (up to $\varepsilon$-stationarity) the
problem arising from \eqref{eq:P} with the additional constraint $z = z^k$.

The numerical results are generated with a prototype MATLAB implementation of \cref{alg:SMIL}, adopting \texttt{intlinprog} for the MILP subproblems arising at \cref{step:SMIL:subproblem} and \texttt{fmincon} for the optional NLP subproblems.%
\footnote{The prototype implementation uses default options except the following: integer tolerance $10^{-6}$, absolute gap, relative gap and constraint tolerances $10^{-8}$ for \texttt{intlinprog}; algorithm \texttt{sqp}, objective gradient provided, max iterations $1000$, optimality and step tolerance $\varepsilon$, and constraint tolerance $10^{-8}$ for \texttt{fmincon}. Code run in MATLAB R2022b, update 10.}
The default algorithmic parameters are set as follows:
tolerance $\varepsilon = 10^{-8}$,
initial radius $\Delta_0 = 1$,
sufficient decrease parameter $\varrho = \nicefrac{1}{10}$,
reduction factor $\kappa = \nicefrac{1}{2}$,
and monotonicity parameter $\meritpfac = \nicefrac{1}{2}$.
\cref{step:SMIL:radiusUpdate} uses the trust region update rule \eqref{eq:TRupdate} with $\varrho_1 = \varrho$ and $\varrho_2 = 2 \varrho$.
The partial localization $\normlp{\cdot}$ uses the $\ell_\infty$-norm, namely $\normlp{x} \coloneqq \max_{i\notin\II} |x_i|$.
If an infeasible starting point $x^0$ is provided, a feasible one is computed by solving $\min_{x\in\XX} \| x - x^0 \|_1$.
Moreover, the execution of \cref{alg:SMIL} is stopped if the MILP solver fails, either by declaring the subproblem at \cref{step:SMIL:subproblem} infeasible or resulting in a negative criticality measure $\psimeas_k$ at \cref{step:SMIL:predictChange}.
These occurrences are both due to numerical issues within the MILP solver; in fact, all subproblems are guaranteed to be feasible and the criticality measure to be nonnegative, by construction, with \emph{exact} MILP solves.

All source codes used to generate and analyze the results in \cref{sec:numresults} have been archived on Zenodo at \href{http://dx.doi.org/10.5281/zenodo.10007957}{\textsc{doi}:~10.5281/zenodo.10007957}.

\subsection{Car with hysteretic turbo charger}

In this section we apply \cref{alg:SMIL} to a numerical example for the point-to-point optimal control problem of a car with turbo charger.
Inspired by \cite[Section V]{nurkanovic2022continuous}, we consider a double-integrator point mass model equipped with a turbo accelerator subject to hysteresis effects.

The car is described by its position $q(t)$, velocity $v(t)$ and turbo state $w(t) \in \{0, 1\}$ for each time $t\in [0,T]$, $T > 0$.
The control variables are the acceleration $a(t)$ and brake $b(t)$ pedals.
The turbo is activated when the velocity exceeds $v \geq v_+ \coloneqq 10$ and is deactivated when it falls below $v \leq v_- \coloneqq 5$; when it is on, it makes the nominal thrust three times more effective.
In contrast, the braking force remains unaffected by the turbo state.
Summarizing, the state vector reads as $(q, v, w) \in \R^3$ with dynamics $\dot{q} = v$, $\dot{v} = f-b$, where the thrust $f$ has two modes of operation, depending on the turbo state, described by $f(a) = a$ if $w=0$ and $f(a) = 3 a$ if $w=1$.
The turbo state $w$ behaves according to the hysteresis characteristic described above.
The acceleration control is bounded by $0 \leq a \leq a_{\max}$, brake control by $0 \leq b \leq b_{\max}$, and the velocity by $|v| \leq v_{\max}$.
The initial state is $q(0) = v(0) = w(0) = 0$.
Given a final time $T>0$ and position $q_{\text{end}} > 0$ and parameters $\alpha_a,\alpha_b \geq 0$, we are interested in minimizing the control effort $\alpha_a \int_0^T a^2(t)\mathrm{d}t + \alpha_b \int_0^T b^3(t)\mathrm{d}t$ to achieve $q(T) = q_{\text{end}}$ and $v(T) = 0$.

By direct discretization
one can formulate this hybrid optimal control problem
in the form \eqref{eq:P},
by rewriting logical propositions as big-M constraints.
Let us consider a time grid $0 = t_0 < \ldots < t_N = T$, with $N$ intervals of size $h = T/N$, and state and control approximations, denoted $q_k,v_k,\ldots,b_k$, such that $q_k \approx q(t_k)$ for each $k=0,\ldots,N$.
Auxiliary variables $f_k$ for the thrust $f$ are included for clarity.
Then, approximated via the trapezoidal rule, the discretized integral objective is
\begin{equation}\label{eq:turbo:cost}
	\minimize\quad
	\alpha_a h \sum_{k=0}^{N-1} \frac{a_k^2 + a_{k+1}^2}{2}
	+
	\alpha_b h \sum_{k=0}^{N-1} \frac{b_k^3 + b_{k+1}^3}{2}
	.
\end{equation}
With finite differences, the discrete-time dynamics read, for $k=0,\ldots,N-1$,
\begin{align}
	&\frac{q_{k+1} - q_k}{h} = \frac{v_{k+1} + v_k}{2} , &
	&\frac{v_{k+1} - v_k}{h} = \frac{f_{k+1} + f_k}{2} - \frac{b_{k+1} + b_k}{2} .
\end{align}
Here, the thrust $f_k$ can be encoded by a linear big-M model, instead of \emph{if-else} cases, as
\begin{align}
	&f_k - a_k \leq{} M w_k ,&
	&f_k - a_k \geq{} - M w_k ,\\
	&f_k - 3 a_k \leq{} M (1-w_k) ,&
	&f_k - 3 a_k \geq{} - M (1-w_k) \nonumber
\end{align}
where $M>0$ is fixed and large enough.
The hysteresis characteristic corresponds to logical propositions describing activation and deactivation conditions, in the form
\begin{align*}
	&w_k = 0 ~\wedge~ v_k > v_+ ~\Rightarrow~ w_{k+1} = 1 ,&
	&w_k = 1 ~\wedge~ v_k < v_- ~\Rightarrow~ w_{k+1} = 0 ,\\
	&w_k = 0 ~\wedge~ v_k < v_+ ~\Rightarrow~ w_{k+1} = 0 ,&
	&w_k = 1 ~\wedge~ v_k > v_- ~\Rightarrow~ w_{k+1} = 1 \nonumber
\end{align*}
for $k=0,\ldots,N-1$, with $\vee$ denoting the logical ``and''.
Rewriting as clauses and using De Morgan's laws, these can be cast as algebraic big-M constraints, reading
\begin{align}
	&v_k \leq v_+ + M (w_k + w_{k+1}) ,&
	&v_k \geq v_- - M (2 - w_k - w_{k+1}) ,\\
	&v_k \geq v_+ - M (w_k + 1 - w_{k+1}) ,&
	&v_k \leq v_- + M (1 - w_k + w_{k+1}) .\nonumber
\end{align}
Finally, bounds are imposed only on the time grid, analogously to boundary conditions, as
\begin{align}\label{eq:turbo:bounds}
	&0 \leq a_k \leq a_{\max} ,&
	&0 \leq b_k \leq b_{\max} ,&
	&- v_{\max} \leq v_k \leq v_{\max} &
	&\text{for}~k=0,\ldots,N ,\\
	&q_0 = v_0 = w_0 = 0 ,&
	&q_N = q_{\text{end}} ,&
	&v_N = 0 .\label{eq:turbo:bounds2}
\end{align}
Overall, a discretization with $N$ intervals results in a problem \eqref{eq:P} according to \eqref{eq:turbo:cost}--\eqref{eq:turbo:bounds2} with
$5(N+1)$ real-valued and $(N+1)$ binary-valued  decision variables,
$2N$ linear equality constraints,
$4(N+1) + 4N$ linear inequality constraints,
and simple bounds.

\paragraph*{Simulations}
We consider instances with $N\in\{25,50,100\}$ discretization intervals and
parameters $\alpha_a=1$, $\alpha_b=10^{-2}$, $q_{\text{end}}=150$, $T=10$, $a_{\max} =5$, $b_{\max}=10$, $v_{\max} = 25$, and $M=20$.
For each instance we run \cref{alg:SMIL} starting from 100 initial guesses obtained by sampling a normal distribution with zero mean and standard deviation 10.

\Cref{alg:SMIL} successfully returned in all cases with a critical point (within the specified tolerance).
Solutions are depicted as trajectories in \cref{fig:turbo} and computational results are summarized in \cref{tab:turbo}.
Out of 100 runs, at most 4 different solutions are returned by \cref{alg:SMIL} for each instance.
These have essentially the same pattern but slightly perturbed turbo activation point,
as shown by \cref{fig:turbo} for $N=100$.
Moreover, \cref{tab:turbo} highlights that not only all solutions have very similar objective value, but also that the solution process is relatively consistent.
In particular, the number of iterations and MILP subproblems are in narrow ranges, independent on the discretization level, whereas the runtime increases with finer discretizations, mostly because MILP subproblems become larger in size.
These observations are illustrated in \cref{fig:turbo_runtime}, which also shows that lower objective values are attained with finer discretizations.
Overall, despite the problem size (up to $n=606$ when $N=100$) and our straightforward implementation without safeguards for numerical errors,
\cref{alg:SMIL} was able to achieve accurate results in a few seconds.

\begin{figure}[tb]
	\centering%
	\input{tikz/turbo.tikz}%
	\caption{Solutions for the problem of a turbo car with $N=100$ discretization intervals: superimposed states and control trajectories of all test runs, with thresholds (black solid lines) and different markers when the turbo is active (red circle) or inactive (blue dot).}%
	\label{fig:turbo}%
\end{figure}

\begin{table}[tb]
	\centering%
	\caption{%
		Summary of results for the problem of a turbo car with $N\in\{25,50,100\}$ discretization intervals.
		Statistics (minimum, maximum, and quartiles over different initial guesses) for the final objective value, runtime, and number of iterations and MILP subproblems.
		\cref{alg:SMIL}'s runtime does not include the polyhedral projection of infeasible initial guesses, provided for comparison.
	}%
	\label{tab:turbo}%
	\begin{tabular}{cc|ccccc}
		& & min & q(25\%) & median & q(75\%) & max \\
		\midrule
		$N=25$ & objective value & 71.20 & 71.20 & 71.20 & 71.20 & 74.35 \\
		& \# iterations & 59 & 61 & 63 & 64 & 69 \\
		& \# MILPs & 125 & 130 & 131 & 133 & 141 \\
		& runtime [s] & 0.44 & 0.58 & 0.62 & 0.67 & 0.91 \\
		initial projection & runtime [s] & 0.06 & 0.13 & 0.15 & 0.18 & 0.29 \\
		\midrule
		$N=50$ & objective value & 69.46 & 69.80 & 69.80 & 69.80 & 71.74 \\
		& \# iterations & 60 & 65 & 66 & 68 & 84 \\
		& \# MILPs & 128 & 134 & 136 & 137 & 163 \\
		& runtime [s] & 1.01 & 1.30 & 1.37 & 1.44 & 1.93 \\
		initial projection & runtime [s] & 0.26 & 0.56 & 0.66 & 0.76 & 1.71 \\
		\midrule
		$N=100$ & objective value & 68.66 & 68.95 & 69.6 & 69.6 & 69.6 \\
		& \# iterations & 66 &68 & 70 & 71 & 73 \\
		& \# MILPs & 134 & 136 & 137 & 139 & 141 \\
		& runtime [s] & 2.84 & 3.26 & 3.43 & 3.53 & 4.34 \\
		initial projection & runtime [s] & 1.40 & 3.03 & 3.82 & 7.49 & 124.08
	\end{tabular}%
\end{table}

\begin{figure}[tb]
	\centering%
	\input{tikz/turbo_runtime.tikz}%
	\caption{Comparison for the problem of a turbo car with $N\in\{25,50,100\}$ discretization intervals. Objective value, runtime and number of MILP subproblems for all test runs.}%
	\label{fig:turbo_runtime}%
\end{figure}

\subsection{Processing network design and operation}

In this section we consider the synthesis of a processing network that was originally formulated in \cite[Example 2]{tuerkay1996logic},
see also \cite{grossmann2013systematic}.
The problem entails the concurrent design and operation planning of a networked system, that is, nodes and flows along the edges are jointly optimized.
\cref{fig:network} shows the superstructure which involves the possible selection of 8 interconnected processing units.
With this illustrative example we demonstrate the potential benefits of including NLP refinement steps in \cref{alg:SMIL}
as well as its performance compared to a brute-force combinatorial enumeration coupled with NLP solves.

The original model consists of 25 real-valued $u$ and 8 Boolean variables $Z$
to model flows on the lines and the existence of processing units.
The objective function includes costs for the selected processing units, operating costs and revenues from sales of products.
The formulation involves material balance and flow constraints,
logical propositions,
and disjunctions for consistent physical modeling.
Boolean specifications are transformed into linear algebraic constraints on binary variables $z$.
For instance,
the proposition $Z_3 ~\Rightarrow~ Z_1 \vee Z_2$
is rewritten as the clause $Z_1 \vee Z_2 \vee \neg Z_3$
and then transformed into $z_1 + z_2 \geq z_3$.
Specifications such as $Z_1 \veebar Z_2$, with $\veebar$ the logical ``xor'', are encoded as algebraic linear constraints on the binary-valued variables: $z_1 + z_2 = 1$;
see \cite[Table 4]{tuerkay1996logic}.
Disjunctions involving linear constraints, e.g.,
$Z_5 ~\Rightarrow{}~ u_{15} = 2 u_{16}$
and
$\neg Z_5 ~\Rightarrow{}~ u_{15} = u_{16} = 0$,
are converted into linear big-M constraints.
For disjunctions with nonlinear expressions we introduce auxiliary real-valued variables $w$ and augment the objective with a quadratic penalty term.
For instance,
$Z_1 ~\Rightarrow{}~ \varphi_1(u) \leq 0$
becomes
$w_1 \leq{} M (1 - z_1)$,
for some large enough $M > 0$,
with the nonlinear term $\lambda [ \varphi_1(u) - w_1 ]^2$ added to the objective function, for some large penalty parameter $\lambda > 0$.

\paragraph*{Simulations}
Consider the problem of \cite[Example 2]{tuerkay1996logic}, reformulated as discussed above with parameters $\lambda=M=10^3$;
in the form \eqref{eq:P} there are 30 real-valued and 8 binary-valued decision variables, 11 linear equality constraints, 37 linear inequality constraints, and simple bounds.
We run \cref{alg:SMIL} on this instance starting from 100 initial guesses obtained by sampling a standard normal distribution;
we then execute the variant of \cref{alg:SMIL} with NLP refinement from the same initial guesses.
To stop the execution, we limit the number of successful iterations (that is, actual updates) to 1000.
Finally, for comparing with a tree search approach, we enumerate all $2^8 = 256$ integer combinations and test their compatibility with $\XX$ (by checking feasibility of the associated fixed-integer LP).
For each of the only $8$ feasible combinations, we solve up to $\varepsilon$-stationarity the corresponding fixed-integer NLP, starting from 10 random initial guesses and picking the best solution to obtain a baseline objective value.

\cref{alg:SMIL} does not solve any of the 100 runs to the specified tolerance, hitting the iteration limit in 96 cases and otherwise generating a negative criticality measure.
In contrast, the variant with NLP refinement returns successfully in 69 cases; the MILP solver declares the subproblem infeasible in 8 cases and generates a negative criticality measure in the remaining 23 cases.
The best solution (in terms of objective value) is visualized in \cref{fig:network}, while computational results are summarized in \cref{tab:network}.
Furthermore, \cref{alg:SMIL} performs similarly with the looser tolerance $\varepsilon=10^{-6}$, whereas the variant with NLP refinement returns successfully in 98 cases, with the MILP solver generating a negative criticality measure in the two remaining.

Without NLP refinement, the algorithm was not able to reach an approximate critical point within the tolerance and the iteration limit. 
Behaving like a steepest descent method, the plain \cref{alg:SMIL} suffers from the slow tail convergence that is typical of first-order methods.
Our numerical observations suggest that the integer components often settle on their final value after a few iterations, then leaving the real-valued components to progress steadily, albeit slowly, toward criticality.
In contrast, despite the computational cost of additional NLP solves, the algorithm can converge quickly when using second-order information.
It is particularly important, however, that the integer-valued variables identify soon and correctly the active set.
The relatively low number of iterations reported in \cref{tab:network} indicate the potential benefits of NLP refinement steps to speed up convergence, as illustrated in \cref{fig:network_runtime}.

\begin{figure}[tb]
	\centering%
	\input{tikz/network.tikz}%
	\caption{Superstructure for the processing network example: units (boxes) and edges (lines) selected (solid) and discarded (dashed) by the optimal design.}%
	\label{fig:network}%
\end{figure}

\begin{table}[tb]
	\centering%
	\caption{Summary of computational results for the problem of a processing network.
		Statistics (minimum, maximum, and quartiles over different initial guesses) for the final objective value, runtime, and number of iterations, MILP and NLP subproblems.
		\cref{alg:SMIL}'s runtime includes NLP solves but not the polyhedral projection of infeasible initial guesses, provided for comparison.}%
	\label{tab:network}%
	\begin{tabular}{cc|ccccc}
		& & min & q(25\%) & median & q(75\%) & max \\
		\midrule
		without NLP & objective value 	& 76.92 & 80.31 & 90.62 & 103.60 & 116.39 \\
		refinement & \# iterations 			& 186 & 1000\textsuperscript{$\bowtie$} & 1000\textsuperscript{$\bowtie$} & 1000\textsuperscript{$\bowtie$} & 1000\textsuperscript{$\bowtie$} \\
		& \# MILPs 		& 353 & 1882 & 1892.5 & 1907.5 & 1933 \\
		& runtime [s] 					& 1.27 & 6.15 & 6.19 & 6.24 & 6.47 \\
		\midrule
		with NLP  & objective value  	& 59.85 & 72.09 & 83.75 & 95.65 & 107.56 \\
		refinement & \# iterations 			& 2 & 12.5 & 18 & 22 & 29 \\
		& \# MILPs 		& 6 & 32.5 & 45.5 & 54 & 73 \\
		& \# NLPs				& 1 & 12 & 16 & 20.5 & 28 \\
		& runtime [s] 					& 0.05 & 0.23 & 0.31 & 0.39 & 0.76 \\
		\midrule
		fixed-integer NLP & objective value 	& 59.85 & 72.09 & 83.70 & 95.65 & 107.56 \\
		& runtime [s] 					& 0.02 & 0.02 & 0.03 & 0.03 & 0.04 \\
		\midrule
		initial projection & runtime [s] 		& 0.01 & 0.01 & 0.01 & 0.01 & 0.09
	\end{tabular}%
	\footnotetext{The symbol \textsuperscript{$\bowtie$} stands for a limit reached.}
\end{table}

\begin{figure}[tb]
	\centering%
	\input{tikz/network_runtime.tikz}%
	\caption{Comparison of results for the problem of a processing network. Objective value, criticality measure, runtime and number of MILP subproblems for all test runs. Runtimes of fixed-integer NLP solves refer only to feasible integer combinations and does not include the preliminary feasibility check. Test runs returning with negative criticality are omitted in the lower plots.}%
	\label{fig:network_runtime}%
\end{figure}

Finally, we point out that, in the variant with NLP refinement, \cref{alg:SMIL} always recovered one of the solutions that are obtained by enumerating the binary combinations and solving the associated fixed-integer NLP, thus indicating a good performance in terms of objective value too.
Not including the preliminary step needed to ascertain feasibility, the runtime of these fixed-integer NLP solves is reported in \cref{tab:network,fig:network_runtime} to contrast with that of \cref{alg:SMIL}.
With NLP refinement, \cref{alg:SMIL}'s computation time amounts to that of NLP and MILP solves, hence it hinges on how many integer combinations are explored.
Analogously, typical MINLP techniques (such as nonlinear branch and bound, outer approximation and Benders decomposition)
involve searching on a tree whose nodes correspond to NLPs; as a result, their performance depends on the branching strategy and the relaxation tightness \cite{belotti2013mixed}.
How these methodologies can be integrated with \cref{alg:SMIL} and investigations on their relative efficiency are left for future assessments.

\section{Final remarks}

\fancyinitial{T}{he} results in this paper could be extended by integrating acceleration schemes
and exploring the effects of inexact subproblem solves on the convergence.
A connection with other stationarity notions should be established, at least for some prominent settings, such as problems with disjunctive constraints.
Weaker assumptions on the problem structure and stronger optimality concepts could also be examined, for instance replacing $\psimeas(x;\Delta)$ with $\psimeas(x;\Delta)/\Delta$ as criticality measure, see \cref{lem:noncritical}.
Future research may focus on active-set warm-starting and lazy constraints for computational efficiency, taking advantage of the similarity of successive inner MILPs.
The proposed approach could also be adopted to tackle problems involving additional nonlinear constraints,
by integrating it within sequential partially-constrained optimization schemes,
such as penalty and barrier methods.

\subsection*{Acknowledgement}
\TheAcknowledgementsADM

\phantomsection
\addcontentsline{toc}{section}{References}%
\bibliographystyle{plain}
\bibliography{biblio}

\end{document}

%% file: preamble.tex
\usepackage{amsmath}
\usepackage{amssymb,amsfonts}
\usepackage{amsthm}
\usepackage{nicefrac}
\usepackage{enumitem}

\usepackage{graphicx}
\usepackage{xcolor}
\definecolor{darkred}{RGB}{153,0,0}
\definecolor{darkblue}{RGB}{0,0,153}
\definecolor{darkgreen}{RGB}{0,153,0}

\usepackage{centernot}
\newcommand{\notimplies}{\centernot\implies}

\usepackage[colorlinks,
	citecolor=darkgreen,
	linkcolor=darkblue,
	urlcolor=darkred]{hyperref}

\newcommand{\midrule}{\hline}

\usepackage{Zallman}
\usepackage{lettrine}

\newcommand{\fancyinitial}[2]{\lettrine[lines=3]{#1}{#2}}

\usepackage[ruled,vlined,linesnumbered,algosection,algo2e]{algorithm2e} 
\newcommand{\algcomment}[1]{\hfill{\color{gray}//~#1}}
\usepackage[capitalize,noabbrev,nameinlink]{cleveref} 

\theoremstyle{plain}
\newtheorem{theorem}{Theorem}[section]

\newtheorem{lemma}[theorem]{Lemma}
\newtheorem{proposition}[theorem]{Proposition}

\theoremstyle{definition}
\newtheorem{definition}[theorem]{Definition}
\theoremstyle{remark}
\newtheorem{remark}[theorem]{Remark}
\newtheorem{examplex}[theorem]{Example}

\newenvironment{example}{\pushQED{\qed}\examplex}{\popQED\endexamplex}

\newcommand{\KeywordsAnd}{\and}
\newcommand{\KeywordsEnd}{}
\newcommand{\keywords}[1]{\par\noindent{\small\def\and{\unskip,\ }{\bf Keywords. }#1.}\par}
\newcommand{\subclass}[1]{\par\noindent{\small\def\and{\unskip,\ }{\bf AMS MSC. }#1.}\par}
\newcommand{\AMSand}{\and}

\crefname{line}{Step}{Steps} 
\SetKw{KwOr}{or}
\SetKw{KwElse}{else}
\SetKw{KwAnd}{and}
\SetKw{KwReturn}{return}
\SetKw{KwGoTo}{go to}

\newcommand{\emailLink}[1]{\textsc{email} \href{mailto:#1}{#1}}
\newcommand{\orcidLink}[1]{\textsc{orcid} \href{https://orcid.org/#1}{#1}}
\newcommand{\amsmscLink}[1]{\href{http://www.ams.org/mathscinet/msc/msc2020.html?t=#1}{#1}}

\newcommand{\N}{\mathbb{N}}
\newcommand{\Z}{\mathbb{Z}}
\newcommand{\R}{\mathbb{R}}

\DeclareMathOperator*{\argmin}{\arg\min}

\DeclareMathOperator{\dom}{dom}
\DeclareMathOperator{\proj}{proj}

\newcommand{\coloneqq}{:=}
\DeclareMathOperator*{\minimize}{minimize}

\DeclareMathOperator{\stt}{subject~to}
\DeclareMathOperator{\wrt}{over}

\newcommand{\func}[3]{#1 \colon #2 \to #3}
\newcommand{\innprod}[2]{\langle #1, #2 \rangle}
\newcommand{\normalcone}{\mathcal{N}}
\newcommand{\normalconefrechet}{\normalcone^{\textup{F}}}

\usepackage{tcolorbox}
\newtcolorbox{mybox}[1][]{%
	left=0pt,
	right=0pt,
	top=0pt,
	bottom=0pt,
	colback=gray!12,
	colframe=gray!12,
	width=\dimexpr\textwidth\relax,
	enlarge left by=0mm,
	boxsep=5pt,
	arc=5pt,outer arc=5pt,
	#1
}
\newtcolorbox{mydefbox}[1][]{%
	left=0pt,
	right=0pt,
	top=0pt,
	bottom=0pt,
	colback=cb2red!16,
	colframe=cb2red!16,
	width=\dimexpr\textwidth\relax,
	enlarge left by=0mm,
	boxsep=5pt,
	arc=5pt,outer arc=5pt,
	#1
}

\newcommand{\norm}[1]{\| #1 \|}
\newcommand{\normlp}[1]{\norm{ #1 }_{\textup{PL}}}
\newcommand{\lpball}{\mathbb{B}_{\textup{PL}}}
\newcommand{\meritf}{m}
\newcommand{\meritpfac}{\kappa_{\meritf}}

\newcommand{\spaceX}{\R^n}
\newcommand{\XX}{\mathcal{X}}

\newcommand{\II}{\mathcal{I}}
\newcommand{\psimeas}{\Psi}

\usepackage{tikz}
\usepackage{pgfplots}
\pgfplotsset{compat=1.18}
\usetikzlibrary{arrows.meta}
\usetikzlibrary {positioning}

%% file: tikz/turbo.tikz
\begin{tikzpicture}
	
	\begin{axis}[%
		width=0.25\columnwidth,
		height=0.23\columnwidth,
		at={(0.0,0.0)},
		scale only axis,
		xmin=0,
		xmax=10,
		xlabel style={font=\color{white!15!black}},
		xlabel={$t$},
		ymin=0,
		ymax=150,
		ylabel style={font=\color{white!15!black},yshift=-5pt},
		ylabel={$q$},
		axis background/.style={fill=white},
		axis x line*=bottom,
		axis y line*=left
		]
		\addplot [color=black, forget plot]
		table[row sep=crcr]{%
			0	150\\
			10	150\\
		};
		\addplot [color=blue, only marks, mark=*, mark size=1, mark options={solid, blue}, forget plot]
		table[]{tikz/data/turbo_N100_100instances_states/turbo_N100_100instances_states-1.tsv};
		\addplot [color=red, only marks, mark=o, mark options={solid, red}, forget plot]
		table[]{tikz/data/turbo_N100_100instances_states/turbo_N100_100instances_states-2.tsv};
	\end{axis}
	
	\begin{axis}[%
		width=0.25\columnwidth,
		height=0.23\columnwidth,
		at={(0.34\columnwidth,0)},
		scale only axis,
		xmin=0,
		xmax=10,
		xlabel style={font=\color{white!15!black}},
		xlabel={$t$},
		ymin=0,
		ymax=27,
		ylabel style={font=\color{white!15!black},yshift=-5pt},
		ylabel={$v$},
		axis background/.style={fill=white},
		axis x line*=bottom,
		axis y line*=left
		]
		\addplot [color=black, forget plot]
		table[row sep=crcr]{%
			0	5\\
			10	5\\
		};
		\addplot [color=black, forget plot]
		table[row sep=crcr]{%
			0	10\\
			10	10\\
		};
		\addplot [color=black, forget plot]
		table[row sep=crcr]{%
			0	25\\
			10	25\\
		};
		\addplot [color=blue, only marks, mark=*, mark size=1, mark options={solid, blue}, forget plot]
		table[]{tikz/data/turbo_N100_100instances_states/turbo_N100_100instances_states-6.tsv};
		\addplot [color=red, only marks, mark=o, mark options={solid, red}, forget plot]
		table[]{tikz/data/turbo_N100_100instances_states/turbo_N100_100instances_states-7.tsv};
	\end{axis}
	
	\begin{axis}[%
		width=0.25\columnwidth,
		height=0.23\columnwidth,
		at={(0.68\columnwidth,0.0)},
		scale only axis,
		xmin=0,
		xmax=10,
		xlabel style={font=\color{white!15!black}},
		xlabel={$t$},
		ymin=-12,
		ymax=14,
		ylabel style={font=\color{white!15!black},yshift=-10pt},
		ylabel={$f-b$},
		axis background/.style={fill=white},
		axis x line*=bottom,
		axis y line*=left
		]
		\addplot [color=black, forget plot]
		table[row sep=crcr]{%
			0	5\\
			10	5\\
		};
		\addplot [color=black, forget plot]
		table[row sep=crcr]{%
			0	0\\
			10	0\\
		};
		\addplot [color=black, forget plot]
		table[row sep=crcr]{%
			0	-10\\
			10	-10\\
		};
		\addplot [color=blue, only marks, mark=*, mark size=1, mark options={solid, blue}, forget plot]
		table[]{tikz/data/turbo_N100_100instances_states/turbo_N100_100instances_states-11.tsv};
		\addplot [color=red, only marks, mark=o, mark options={solid, red}, forget plot]
		table[]{tikz/data/turbo_N100_100instances_states/turbo_N100_100instances_states-12.tsv};
	\end{axis}
\end{tikzpicture}%

%% file: tikz/turbo_runtime.tikz
\definecolor{mycolor_N25}{rgb}{0.99325,0.90616,0.14394}%
\definecolor{mycolor_N50}{rgb}{0.12815,0.56511,0.55089}%
\definecolor{mycolor_N100}{rgb}{0.26700,0.00487,0.32942}%

\begin{tikzpicture}

\begin{axis}[%
width=0.36\columnwidth,
height=0.23\columnwidth,
at={(0,0)},
scale only axis,
xmin=68,
xmax=75,
xlabel style={font=\color{white!15!black}},
xlabel={objective value},
ymode=log,
ymin=0.4,
ymax=10,
yminorticks=true,
ylabel style={font=\color{white!15!black}},
ylabel={runtime [s]},
axis background/.style={fill=white},
legend style={legend cell align=left, align=left, draw=none, anchor=south, at={(1,1.05)}},
legend style={/tikz/every even column/.append style={column sep=0.5cm}},
legend columns=-1
]
\addplot [color=mycolor_N25, only marks, mark size=1.9pt, mark=*, mark options={solid, mycolor_N25}]
  table[]{tikz/data/turbo_runtime/turbo_runtime-1.tsv};
\addlegendentry{$N=25$}

\addplot [color=mycolor_N50, only marks, mark size=1.7pt, mark=*, mark options={solid, mycolor_N50}]
  table[]{tikz/data/turbo_runtime/turbo_runtime-2.tsv};
\addlegendentry{$N=50$}

\addplot [color=mycolor_N100, only marks, mark size=1.5pt, mark=*, mark options={solid, mycolor_N100}]
  table[]{tikz/data/turbo_runtime/turbo_runtime-3.tsv};
\addlegendentry{$N=100$}

\end{axis}

\begin{axis}[%
width=0.36\columnwidth,
height=0.23\columnwidth,
at={(0.39\columnwidth,0)},
scale only axis,
xmin=120,
xmax=170,
xmode=log,
xtick={120,130,140,150,160,170},
xticklabels={,130,140,150,160,},
xlabel style={font=\color{white!15!black}},
xlabel={\# MILPs},
ymode=log,
ymin=0.4,
ymax=10,
yminorticks=true,
yticklabels={},
axis background/.style={fill=white},
legend style={legend cell align=left, align=left, draw=none}
]
\addplot [color=mycolor_N25, only marks, mark size=1.9pt, mark=*, mark options={solid, mycolor_N25}]
  table[]{tikz/data/turbo_runtime/turbo_runtime-4.tsv};

\addplot [color=mycolor_N50, only marks, mark size=1.7pt, mark=*, mark options={solid, mycolor_N50}]
  table[]{tikz/data/turbo_runtime/turbo_runtime-5.tsv};

\addplot [color=mycolor_N100, only marks, mark size=1.5pt, mark=*, mark options={solid, mycolor_N100}]
  table[]{tikz/data/turbo_runtime/turbo_runtime-6.tsv};

\end{axis}

%
%
%

\end{tikzpicture}%

%% file: tikz/network.tikz
\begin{tikzpicture}[scale=0.8, font=\small, thick]
	\draw[dashed] (2, 0.4) rectangle (3, 1.2) {};
	\node at (2.5, 0.8) {$z_1$};
	\draw[black] (2, -0.4) rectangle (3, -1.2);
	\node at (2.5, -0.8) {$z_2$};
	\draw[black] (6, -3.6) rectangle (7, -2.8);
	\node at (6.5, -3.2) {$z_3$};
	\draw[black] (6, 0.4) rectangle (7, 1.2);
	\node at (6.5, 0.8) {$z_4$};
	\draw[dashed] (6, -2.0) rectangle (7, -1.2);
	\node at (6.5, -1.6) {$z_5$};
	\draw[black] (9, 1.2) rectangle (10, 2.0);
	\node at (9.5, 1.6) {$z_6$};
	\draw[dashed] (9, -0.4) rectangle (10, 0.4);
	\node at (9.5, 0) {$z_7$};
	\draw[black] (9, -2.0) rectangle (10, -1.2);
	\node at (9.5, -1.6) {$z_8$};
	%
	%
	%
	\draw[->,black] (0.125, 0) to (1, 0);
	\node[above] at (0.5, 0) {$u_1$};
	\draw (0,0) circle [radius=0.125];
	\draw[->,dashed] (1, 0) -- (1, 0.8) -- (2, 0.8);
	\node[above] at (1.5, 0.8) {$u_2$};
	\draw[->,dashed] (3, 0.8) -- (4, 0.8) -- (4, -0.8+0.1);
	\draw[-,dashed] (4, -0.8+0.1) to (4, -0.8);
	\node[above] at (3.5, 0.8) {$u_3$};
	\draw[->,black] (1, 0) -- (1, -0.8) -- (2, -0.8);
	\node[above] at (1.5, -0.8) {$u_4$};
	\draw[->,black] (3, -0.8) to (4-0.1, -0.8);
	\draw[-,black] (4-0.1, -0.8) to (4, -0.8);
	\node[above] at (3.5, -0.8) {$u_5$};
	\draw[->,black] (4cm+1pt, -0.8) to (5, -0.8);
	\node[above] at (4.5, -0.8) {$u_{11}$};
	\draw[->,black] (5, -0.8) -- (5, 0.8) -- (6, 0.8);
	\node[above] at (5.5, 0.8) {$u_{12}$};
	\draw[->,dashed] (5, -0.8) -- (5, -1.6) -- (6, -1.6);
	\node[above] at (5.5, -1.6) {$u_{15}$};
	\draw[->,black] (4, -0.8) to (4, -3.2);
	\node[left] at (4, -2.0) {$u_6$};
	\draw[->,black] (4, -3.2) to (4, -4.0);
	\node[left] at (4, -3.6) {$u_7$};
	\draw (4,-4.0-0.125) circle [radius=0.125];
	\draw[->,black] (4, -3.2) to (6, -3.2);
	\node[above] at (5, -3.2) {$u_8$};
	\draw[->,black] (10, 1.6) -- (11, 1.6) -- (11, 0.8+0.1);
	\draw[-,black] (11, 0.8+0.1) to (11, 0.8);
	\node[above] at (10.5, 1.6) {$u_{20}$};
	\draw[->,dashed] (10, 0) -- (11, 0) -- (11, 0.8-0.1);
	\draw[-,dashed] (11, 0.8-0.1) to (11, 0.8);
	\node[above] at (10.5, 0) {$u_{22}$};
	\draw[->,black] (11, 0.8) to (12, 0.8);
	\node[above] at (11.5, 0.8) {$u_{23}$};
	\draw[->,black] (12, 0.8) to (12, -0.8+0.125);
	\node[left] at (12, -0.4) {$u_{24}$};
	\draw (12,-0.8) circle [radius=0.125];
	\draw[->,black] (12, 0.8) -- (12, 2.4) -- (6.5, 2.4) -- (6.5, 1.2);
	\node[left] at (6.5, 2.0) {$u_{14}$};
	\draw[->,black] (7, 0.8) to (8, 0.8);
	\node[above] at (7.5, 0.8) {$u_{13}$};
	\draw[->,black] (8, 0.8) -- (8, 1.6) -- (9, 1.6);
	\node[above] at (8.5, 1.6) {$u_{19}$};
	\draw[->,dashed] (8, 0.8) -- (8, 0) -- (9, 0);
	\node[above] at (8.5, 0) {$u_{21}$};
	\draw[->,dashed] (7, -1.6) to (8-0.1, -1.6);
	\draw[-,dashed] (8-0.1, -1.6) to (8, -1.6);
	\node[above] at (7.5, -1.6) {$u_{16}$};
	\draw[->,black] (8, -1.6) to (9, -1.6);
	\node[above] at (8.5, -1.6) {$u_{17}$};
	\draw[->,dashed] (8, -0.8-0.125) to (8, -1.6+0.1);
	\draw[-,dashed] (8, -1.6+0.1) to (8, -1.6);
	\node[above] at (8, -0.8) {$u_{25}$};
	\draw (8,-0.8) circle [radius=0.125];
	\draw[->,black] (7, -2.8-0.25) -- (8, -2.8-0.25) -- (8, -1.6-0.1);
	\draw[-,black] (8, -1.6-0.1) to (8, -1.6);
	\node[above] at (7.5, -2.8-0.25) {$u_9$};
	\draw[->,black] (7, -3.2-0.25) -- (9.5, -3.2-0.25) -- (9.5, -2.0);
	\node[above] at (9, -3.2-0.25) {$u_{10}$};
	\draw[->,black] (10, -1.6) -- (12, -1.6) -- (12, -2.4+0.125);
	\node[above] at (11.5, -1.6) {$u_{18}$};
	\draw (12,-2.4) circle [radius=0.125];
\end{tikzpicture}

%% file: tikz/network_runtime.tikz
\definecolor{mycolor1}{RGB}{27,158,119}%
\definecolor{mycolor2}{RGB}{217,95,2}%
\definecolor{mycolor3}{RGB}{117,112,179}%
\begin{tikzpicture}

\begin{axis}[%
width=0.36\columnwidth,
height=0.23\columnwidth,
at={(0,0.25\columnwidth)},
scale only axis,
xmin=58,
xmax=118,
xticklabels={},
ymode=log,
ymin=0.01,
ymax=10,
yminorticks=true,
ylabel style={font=\color{white!15!black}},
ylabel={runtime [s]},
axis background/.style={fill=white},
legend style={legend cell align=left, align=left, draw=none, anchor=south, at={(1,1.05)}},
legend style={/tikz/every even column/.append style={column sep=0.5cm}},
legend columns=-1
]
\addplot [color=mycolor1, only marks, mark size=1.9pt, mark=*, mark options={solid, mycolor1}]
  table[]{tikz/data/network_runtime/network_runtime-1.tsv};
\addlegendentry{without NLP refinement}

\addplot [color=mycolor2, only marks, mark size=1.7pt, mark=*, mark options={solid, mycolor2}]
  table[]{tikz/data/network_runtime/network_runtime-2.tsv};
\addlegendentry{with NLP refinement}

\addplot [color=mycolor3, only marks, mark size=1.5pt, mark=triangle*, mark options={solid, mycolor3}]
  table[]{tikz/data/network_runtime/network_runtime-3.tsv};
\addlegendentry{fixed-integer NLP}

\end{axis}

\begin{axis}[%
width=0.36\columnwidth,
height=0.23\columnwidth,
at={(0.39\columnwidth,0.25\columnwidth)},
scale only axis,
xmode=log,
xmin=5,
xmax=3000,
xticklabels={},
ymode=log,
ymin=0.01,
ymax=10,
yminorticks=true,
yticklabels={},
axis background/.style={fill=white},
legend style={legend cell align=left, align=left, draw=white!15!black}
]
\addplot [color=mycolor1, only marks, mark size=1.9pt, mark=*, mark options={solid, mycolor1}]
  table[]{tikz/data/network_runtime/network_runtime-4.tsv};

\addplot [color=mycolor2, only marks, mark size=1.7pt, mark=*, mark options={solid, mycolor2}]
  table[]{tikz/data/network_runtime/network_runtime-5.tsv};
\end{axis}

\begin{axis}[%
width=0.36\columnwidth,
height=0.23\columnwidth,
at={(0,0)},
scale only axis,
unbounded coords=jump,
xmin=58,
xmax=118,
xlabel style={font=\color{white!15!black}},
xlabel={objective value},
ymode=log,
ymin=1e-10,
ymax=1,
yminorticks=true,
ylabel style={font=\color{white!15!black}},
ylabel={criticality},
ytick={1e-2,1e-5,1e-8},
axis background/.style={fill=white},
legend style={legend cell align=left, align=left, draw=white!15!black}
]
\addplot [color=mycolor1, only marks, mark size=1.9pt, mark=*, mark options={solid, mycolor1}]
  table[]{tikz/data/network_runtime/network_runtime-6.tsv};

\addplot [color=mycolor2, only marks, mark size=1.7pt, mark=*, mark options={solid, mycolor2}]
  table[]{tikz/data/network_runtime/network_runtime-7.tsv};
\end{axis}

\begin{axis}[%
width=0.36\columnwidth,
height=0.23\columnwidth,
at={(0.39\columnwidth,0)},
scale only axis,
unbounded coords=jump,
xmode=log,
xmin=5,
xmax=3000,
xlabel style={font=\color{white!15!black}},
xlabel={\# MILPs},
ymode=log,
ymin=1e-10,
ymax=1,
yminorticks=true,
ytick={1e-2,1e-5,1e-8},
yticklabels={},
axis background/.style={fill=white},
legend style={legend cell align=left, align=left, draw=white!15!black}
]
\addplot [color=mycolor1, only marks, mark size=1.9pt, mark=*, mark options={solid, mycolor1}]
  table[]{tikz/data/network_runtime/network_runtime-8.tsv};

\addplot [color=mycolor2, only marks, mark size=1.7pt, mark=*, mark options={solid, mycolor2}]
  table[]{tikz/data/network_runtime/network_runtime-9.tsv};
\end{axis}

\end{tikzpicture}%